\documentclass[reqno,twoside,12pt]{article}
 \usepackage[english]{babel}
 \usepackage{times}
 \usepackage[
    paper=a4paper,
    portrait=true,
    textwidth=425pt,
    textheight=650pt,
    tmargin=3cm,
    marginratio=1:1
            ]{geometry}
 \usepackage[all]{xy}       
 \usepackage[centertags]{amsmath}
 \usepackage{amsfonts}
 \usepackage{amsthm}
 \usepackage{amssymb}
 \usepackage{enumerate}
 \usepackage{mathrsfs}      
 \usepackage{color}
 \usepackage{arydshln}      

 \theoremstyle{plain}
 \newtheorem{Thm}{Theorem}
 
 \newtheorem{Lemma}[Thm]{Lemma}
 \newtheorem{Prop}[Thm]{Proposition}

 \theoremstyle{definition}
\def\text#1{\;\;\;\;{\rm \hbox{#1}}\;\;\;\;}
\def\qquad{\quad\quad}

\def\msy#1{{\mathbb #1}}
\def\C{{\msy C}}

\def\D{{\msy D}}

\def\ga{\alpha}

\def\gf{\varphi}

\def\gl{\lambda}

\def\gs{\sigma}

\def\gS{\Sigma}

\def\fa{{\mathfrak a}}

\def\fg{{\mathfrak g}}
\def\fh{{\mathfrak h}}

\def\fk{{\mathfrak k}}

\def\fm{{\mathfrak m}}
\def\fn{{\mathfrak n}}

\def\fp{{\mathfrak p}}
\def\fq{{\mathfrak q}}

\def\to{\rightarrow}
\def\Re{\mathrm{Re}\,}

\def\End{\mathrm{End}}

\def\Hom{\mathrm{Hom}}

\def\ad{\mathrm{ad}}

\def\tr{\mathrm{tr}\,}
\DeclareMathOperator*{\Res}{Res}
\def\Restau{\mathrm{Res}_{\tau}}
\def\Restriv{\mathrm{Res}_{\1_{K}}}
\def\spn{\mathrm{span}}

\def\cA{{\mathcal A}}

\def\cC{{\mathcal C}}

\def\cH{{\mathcal H}}
\def\cI{{\mathcal I}}

\def\cP{{\mathcal P}}

\def\cR{{\mathcal R}}

\def\cW{{\mathcal W}}

\def\oC{\cA_{M,2}}

\def\Lie{\mathop{\rm Lie}}

\def\Cartan{\theta}

\def\Vtau{V_\tau}

\def\dotvar{\, \cdot\,}

\def\Ht{\cH}

\def\It{\cI}

\def\Rt{\cR}

\def\Cen{\mathcal{Z}}
\def\Nor{\mathcal{N}}

\def\1{\mathbf{1}}

\def\cusp{\mathrm{cusp}}
\def\res{\mathrm{res}}
\def\ds{\mathrm{ds}}

 \title{$K$-invariant cusp forms for reductive symmetric spaces of split rank one}
 \author{Erik~P.~van den Ban,
         Job~J.~Kuit\footnote{Funded by the Deutsche Forschungsgemeinschaft (DFG, German Research Foundation) -- 262362164.}\ \ and
         Henrik~Schlichtkrull}
\date{}

\begin{document}
 \maketitle
 \vspace{-15pt}
\begin{abstract}
Let $G/H$ be a reductive symmetric space of split rank $1$ and let $K$ be a maximal compact subgroup of $G$. In a previous article the first two authors introduced a notion of cusp forms for $G/H$. We show that the space of cusp forms coincides with the closure of the $K$-finite generalized matrix coefficients of discrete series representations if and only if there exist no $K$-spherical discrete series representations. Moreover, we prove that every $K$-spherical discrete series representation occurs with multiplicity $1$ in the Plancherel decomposition of $G/H$.
\end{abstract}
\section*{Introduction}\addcontentsline{toc}{section}{Introduction}
By refining a suggestion of M. Flensted-Jensen the first two authors introduced a notion of cusp forms for reductive symmetric spaces of split rank $1$ in \cite{vdBanKuit_HC-TransformAndCuspForms}. For reductive groups of split rank $1$ this definition of cusp forms coincides with Harish-Chandra's definition. It further generalizes the definition of cusp forms for hyperbolic spaces given in \cite{AndersenFlenstedJensenSchlichtkrull_CuspidalDiscreteSerieseForSemisimpleSymmetricSpaces} and \cite{AndersenFlensted-Jensen_CuspidalDiscreteSeriesForProjectiveHyperbolicSpaces}. The definition of cusp forms does not straightforwardly generalize to reductive symmetric spaces of higher split rank as the cuspidal integrals are not always convergent, see \cite[Section 4]{FlenstedJensenKuit_CuspidalIntegrals}.

Let $G/H$ be a reductive symmetric space of split rank one. We write $\cC(G/H)$ for the space of Harish-Chandra Schwartz functions on $G/H$. In \cite{vdBanKuit_HC-TransformAndCuspForms} a class $\cP_{\fh}$ of minimal parabolic subgroups is identified such that the cuspidal integrals
$$
\Rt_Q\phi (g):= \int_{N_Q/N_Q \cap H} \phi(gn)\; dn        \qquad (g \in G)
$$
are absolutely convergent for every $Q \in \cP_\fh$ and $\phi \in \cC(G/H)$. Here $N_{Q}$ is the unipotent radical of $Q$.
A function $\phi \in \cC(G/H)$ is said to be a cusp form if $\Rt_Q\phi=0$ for all $Q \in \cP_\fh$.
Let $\cC_\cusp(G/H)$ denote the space of cusp forms and let $\cC_{\ds}(G/H)$ be the closure in $\cC(G/H)$ of the span of $K$-finite generalized matrix coefficients of discrete series representations for $G/H$. It is shown in \cite[Theorem 8.20]{vdBanKuit_HC-TransformAndCuspForms} that
$$
\cC_\cusp(G/H) \subseteq \cC_\ds(G/H).
$$

Let $K$ be a maximal compact subgroup of $G$ so that $K\cap H$ is a maximal compact subgroup of $H$. For a finite set $\vartheta$ of irreducible unitary representations of $K$ we write $\cC(G/H)_{\vartheta}$ for the subspace of $\cC(G/H)$ of $K$ finite functions with $K$-isotypes contained in $\vartheta$. In \cite[Theorem 8.4]{vdBanKuit_HC-TransformAndCuspForms} it is established that
$\cC_{\ds}(G/H)_{\vartheta}:=\cC_{\ds}(G/H)\cap\cC(G/H)_{\vartheta}$ admits an $L^2$-orthogonal decomposition
$$
\cC_\ds(G/H)_{\vartheta} = \cC_\cusp(G/H)_{\vartheta} \oplus \cC_{\res}(G/H)_{\vartheta},
$$
where $\cC_\res(G/H)_{\vartheta}$ is spanned by certain residues of Eisenstein integrals defined in terms of parabolic subgroups in $\cP_{\fh}$.

It is a fundamental result of Harish-Chandra that for reductive Lie groups no residual discrete series representation occur, i.e., if $G$ is a reductive Lie group then
\begin{equation}\label{eq discr series = cusp forms}
\cC_\ds(G)= \cC_\cusp(G).
\end{equation}
See \cite{Harish-Chandra_DiscreteSeriesForSemisimpleLieGroupsII}, \cite[Thm. 10]{Harish-Chandra_HarmonicAnalysisOnSemisimpleLieGroups} and \cite[Sects. 18 \& 27]{Harish-Chandra_HarmonicAnalyisOnRealReductiveGroupsI}; see also \cite[Thm. 16.4.17]{Varadarajan_HarmonicAnalysisOnRealReductiveGroups}. In \cite[Theorem 8.22]{vdBanKuit_HC-TransformAndCuspForms} the following criterion was given for the analogue of (\ref{eq discr series = cusp forms}) for reductive symmetric spaces of split rank $1$,
\begin{equation}
\label{eq criterion ds is cusp}
\cC_\res(G/H)^K = 0 \;  \Rightarrow\;  \cC_\cusp(G/H) = \cC_\ds(G/H).
\end{equation}
The main result of this article is that this is actually an equivalence.

\begin{Thm}\label{Main theorem}
There exist no non-zero $K$-invariant cusp forms, i.e.,
\begin{equation}\label{eq C_ds(G/H)^K cap C_cusp(G/H)=0}
\cC_{\cusp}(G/H)^{K}=\{0\}.
\end{equation}
Moreover, the following are equivalent.
\begin{enumerate}[(i)]
\item
$\cC_{\ds}(G/H)=\cC_{\cusp}(G/H)$;
\item
$\cC_{\ds}(G/H)^{K}=\{0\}$.
\end{enumerate}
\end{Thm}

The analysis needed for the proof of Theorem \ref{Main theorem} is further used to prove the following theorem, which confirms some special cases of the multiplicity one result of \cite{Bien_DModulesAndSphericalRepresentations}, page 3, Theorem 3.

\begin{Thm}\label{Thm multiplicity 1}
Let $G/H$ have split rank $1$.
Every $K$-spherical discrete series representation occurs with multiplicity $1$ in the Plancherel decomposition of $G/H$.
\end{Thm}

The article is organized as follows. We start by introducing the necessary notation in Section \ref{Section Notation}. In Sections \ref{Section tau-spherical cusp forms} and \ref{Section Formula for H_(Q,tau)} we set up the machinery needed for the proof of Theorem \ref{Main theorem}. The proof is given in Section \ref{Section proof}. Finally, Theorem \ref{Thm multiplicity 1} is proved in Section \ref{Section Multiplicity}.

\section{Notation and preliminaries}\label{Section Notation}

Throughout the paper, $G$ will be a reductive Lie group of the Harish-Chandra class, $\sigma$ an involution of $G$ and $H$ an open subgroup of the fixed point subgroup for $\sigma$.
We assume that $H$ is essentially connected as defined in \cite[p.\ 24]{vdBan_ConvexityThm}.
The involution of the Lie algebra $\fg$ of $G$ obtained by deriving $\gs$ is denoted by the same symbol.
Accordingly, we write $\fg=\fh\oplus\fq$ for the decomposition of $\fg$ into the $+1$ and $-1$-eigenspaces for $\sigma$. Thus, $\fh$ is the Lie algebra of $H$. Here and in the rest of the paper, we adopt the convention to denote Lie groups
by Roman capitals, and their Lie algebras by the corresponding Fraktur lower cases.

We fix a Cartan involution $\theta$ that commutes with $\sigma$ and write $\fg=\fk\oplus\fp$ for the corresponding decomposition of $\fg$ into the $+1$ and $-1$ eigenspaces for $\theta$. Let $K$ be the fixed point subgroup of $\theta$. Then $K$ is a $\sigma$-stable maximal compact subgroup with Lie algebra $\fk$.   In addition, we fix a maximal abelian subspace $\fa_{\fq}$ of $\fp\cap\fq$ and a maximal abelian subspace $\fa$ of $\fp$ containing $\fa_{\fq}$. Then $\fa$ is $\sigma$-stable and
$$
\fa=\fa_{\fq}\oplus\fa_{\fh},
$$
where $\fa_{\fh}=\fa\cap \fh$.
This decomposition allows us to identify $\fa_{\fq}^{*}$ and $\fa_{\fh}^{*}$ with the subspaces
$(\fa/\fh)^*$ and $(\fa/\fq)^*$ of $\fa^{*},$ respectively.

Let $A$ be the connected Lie group with Lie algebra $\fa$. We define $M$ to be the centralizer of $A$ in $K$. The  set of minimal parabolic subgroups containing $A$ is denoted by $\cP(A)$.

If $Q$ is a parabolic subgroup, then its nilpotent radical will be denoted by
$N_{Q}.$ Furthermore, we agree to write $\bar Q = \Cartan Q$ and $\bar N_Q = \Cartan N_Q.$
Note that if $Q\in\cP(A)$, then $MA$ is a Levi subgroup of $Q$ and $Q=MAN_{Q}$ is
the Langlands decomposition of $Q$.

The root system of $\fa$ in $\fg$ is denoted by $\Sigma = \gS(\fg,\fa).$
For $Q\in\cP(A)$ we put
$$
\Sigma(Q): = \{\ga \in \gS :  \fg_\ga \subseteq     \fn_Q\}.
$$
Let $\Cen_\fg(\fa_{\fq})$ denote the centralizer of $\fa_{\fq}$ in $\fg.$
We define the elements $\rho_Q $ and $\rho_{Q,\fh}$ of $\fa^{*}$ by
$$
\rho_Q(\dotvar) = \frac12 \tr(\ad(\dotvar)|_{\fn_Q}), \qquad {\rm and}\quad
\rho_{Q,\fh}(\dotvar) = \frac12 \tr(\ad(\dotvar)|_{\fn_Q \cap \Cen_\fg(\fa_{\fq})}).
$$
We say that $Q$ is $\fh$-compatible if
$$
\langle\alpha,\rho_{Q,\fh}\rangle\geq0\quad \textnormal{for all}\quad\alpha\in\Sigma(Q).
$$
We write $\cP_{\fh}(A)$ for the subset of $\cP(A)$ consisting of all $\fh$-compatible parabolic subgroups.

\section{$\tau$-Spherical cusp forms}\label{Section tau-spherical cusp forms}
Let $(\tau,V_{\tau})$ be a finite dimensional representation of $K$.
We write $C^{\infty}(G/H:\tau)$ for the space of smooth functions $\phi:G/H\to V_{\tau}$ satisfying the transformation rule
$$
\phi(kx)
=\tau(k)\phi(x)\qquad(k\in K,x\in G/H)
$$
and we write $\cC(G/H:\tau)$ for the space of $\phi\in C^{\infty}(G/H:\tau)$ that are Schwartz. (See \cite[Section 3.1]{vdBanKuit_HC-TransformAndCuspForms}.)

Let $W(\fa_{\fq})$ be the Weyl group of the root system of $\fg$ in $\fa_{\fq}$. Then $W(\fa_{\fq})$ can be realized as the quotient $W(\fa_{\fq})=\Nor_{K}(\fa_{\fq})/\Cen_{K}(\fa_{\fq})$. Let $W_{K\cap H}(\fa_{\fq})$ be the subgroup of $W(\fa_{\fq})$ of elements that can be realized in $\Nor_{K\cap H}(\fa_{\fq})$.
We choose a set $\cW$ of representatives of $W(\fa_{\fq})/W_{K\cap H}(\fa_{\fq})$ in $\Nor_{K}(\fa_{\fq})\cap\Nor_{K}(\fa_{\fh})$ such that $e \in \cW$. This is possible because of the identity
$$
\Nor_{K}(\fa_{\fq})
=\big(\Nor_{K}(\fa_{\fq})\cap \Nor_{K}(\fa_{\fh})\big)\Cen_{K}(\fa_{\fq}).
$$
See \cite{Rossmann_TheStructureOfSemisimpleSymmetricSpaces} at the top of page 165.

Let
$$
\fa_{0}
:=\bigcap_{\alpha\in\Sigma\cap \fa_{\fh}^{*}}\ker(\alpha)
$$
and define
$$
\fm_{0}
:=\Cen_{\fg}(\fa_{\fq})\cap\fa_{0}^{\perp}.
$$
Let $\fm_{0n}$ be the direct sum of all non-compact ideals in $\fm_{0}$ and let $M_{0n}$ be the connected subgroup of $G$ with Lie algebra $\fm_{0n}$.
We define $\tau_{M}$ to be the restriction of $\tau$ to $M$ and write $\tau_{M}^{0}$ for the subrepresentation of $\tau_{M}$ on $(V_{\tau})^{M_{0n}\cap K}$. We further define
$$
\oC(\tau)
:=\bigoplus_{v \in \cW}C^\infty(M/M\cap vHv^{-1}:\tau_{M}^{0}).
$$
We equip $\oC(\tau)$ with the natural Hilbert space structure and note that it is finite dimensional.

Given $v\in\cW$ and $Q\in\cP(A)$ we define the parabolic subgroup $Q^{v}\in\cP(A)$ by
$$
Q^{v}
:= v^{-1}Qv.
$$

Let $Q\in\cP_{\fh}(A)$. For $\phi\in\cC(G/H:\tau)$ define $\Ht_{Q,\tau}\phi:A_{\fq}\to \oC(\tau)$ to be the function given by
$$
\Big(\Ht_{Q,\tau}\phi(a)\Big)_{v}(m)
=a^{\rho_{Q}-\rho_{Q,\fh}}\int_{N_{Q^{v}}/H\cap N_{Q^{v}}}\phi(mavn)\,dn
\qquad(v\in\cW, m\in M, a\in A_{\fq}).
$$
By \cite[Theorem 5.12]{vdBanKuit_HC-TransformAndCuspForms} the integral is absolutely convergent for every $\phi\in\cC(G/H)$. Furthermore, the map $\Ht_{Q,\tau}:\cC(G/H:\tau)\to C^{\infty}(A_{\fq})\otimes\oC(\tau)$ thus obtained is continuous.
We call $\phi\in\cC(G/H:\tau)$ a $\tau$-spherical cusp form if for every $Q\in\cP_{\fh}(A)$
$$
\Ht_{Q,\tau}\phi=0.
$$

We will now describe the relation between the $\tau$-spherical cusp forms and the cusp forms defined in the previous section. Let $\vartheta$ be a finite subset of $\widehat{K}$. For a representation of $K$ on a vector space $V$, we denote the subspace of $K$-finite vectors with isotypes in $\vartheta$ by $V_{\vartheta}$.
Consider $C(K)$ equipped with the left-regular representation of $K$. Define $V_{\tau}:=C(K)_{\vartheta}$, i.e., let $V_{\tau}$ be the space of $K$-finite functions on $K$, whose isotopy types for the left regular representation are contained in $\vartheta$. We define $\tau$ to be the unitary representation of $K$ on $V_{\tau}$ obtained from the right action.
Then there is a canonical isomorphism
$$
\varsigma:\cC(G/H)_{\vartheta}\to\cC(G/H:\tau)
$$
given by
$$
\varsigma\phi(x)(k)=\phi(kx)
\qquad\big(\phi\in\cC(G/H)_{\vartheta}, k\in K, x\in G/H\big).
$$
By \cite[Remark 6.3]{vdBanKuit_HC-TransformAndCuspForms} we now have
$$
\varsigma\Big(\cC_{\cusp}(G/H)_{\vartheta}\Big)
=\cC_{\cusp}(G/H:\tau).
$$

\section{A formula for $\Ht_{Q,\tau}$}\label{Section Formula for H_(Q,tau)}
In \cite{vdBanKuit_EisensteinIntegrals} Eisenstein integrals were constructed which were then used in \cite{vdBanKuit_HC-TransformAndCuspForms} to derive a formula for $\Ht_{Q,\tau}$. This formula is very useful to analyze the relation between cusp forms and discrete series representations. We will now recall this formula and all relevant objects. For details we refer to the two mentioned articles.

We fix $Q\in\cP_{\fh}(A)$. We further choose a minimal $\sigma\theta$-stable parabolic subgroup $P_{0}$ containing $A$, with the property that $\Sigma(Q)\cap\sigma\theta\Sigma(Q)\subseteq\Sigma(P_{0})$. (It is easy to see that such a minimal $\sigma\theta$-stable parabolic subgroup always exists.)

Given $\psi \in \cA_{M,2}(\tau)$, $\lambda\in\fa_{\fq\C}^{*}$ and $v\in\cW$ we define
the function  $\psi_{v,Q,\lambda}: G \to V_{\tau}$ by
$$
\psi_{v,Q,\lambda} (kman ) = a^{\lambda-\rho_{Q}-\rho_{Q,\fh}}\,\tau(k) \psi_{v}(m).
$$
Let $\omega_{v}$ be a non-zero density on $\fh/\fh\cap\Lie(v^{-1}Qv)$.
If $-\langle\Re\lambda,\alpha\rangle$ is sufficiently large for every $\alpha\in \Sigma(Q)\cap\sigma\theta\Sigma(Q)$, then for each $x \in G$ and $v\in\cW$ the function
$$
h \mapsto \psi_{Q, \lambda}(xhv^{-1})\;dl_h(e)^{-1*}\omega
$$
defines an integrable $V_{\tau}$-valued density on $H/H\cap v^{-1}Qv$ (see \cite[Proposition 8.2]{vdBanKuit_EisensteinIntegrals}). For these $\lambda$ we define the Eisenstein integral $E_{\tau}(Q:\psi:\lambda)\in C^{\infty}(G/H:\tau)$
by
$$
E_{\tau}(Q:\psi:\lambda)(x)
:=\sum_{v\in\cW} \int_{H/H\cap v^{-1}Qv}\;
 \psi_{v,Q, \lambda}(xhv^{-1})\,dl_h(e)^{-1*}\,\omega_{v}, \quad (x \in G).
$$
The function $\lambda\mapsto E_{\tau}(Q:\psi:\lambda)$ extends to a meromorphic $C^{\infty}(G/H:\tau)$-valued function on $\fa_{\fq\C}^{*}$. This definition of Eisenstein integrals coincides with the definition in \cite[Section 8]{vdBanKuit_EisensteinIntegrals}.
We write $E_{\tau}(Q:\dotvar)$ for the map
$$
\oC(\tau)\ni\psi\mapsto E_{\tau}(Q:\psi:\dotvar).
$$

We define
$$
\fa_{\fq}^{*+}
=\fa_{\fq}^{*+}(P_{0})
:=\big\{\lambda\in\fa_{\fq}^{*}:\langle\lambda,\alpha\rangle>0\text{for all}\alpha\in\Sigma(P_{0})\big\}.
$$
Let $S_{Q,\tau}$ be the set of $\lambda\in\fa_{\fq}^{*+}+i\fa_{\fq}^{*}$ such that $E_{\tau}(Q:-\dotvar)$ is singular at $\lambda$. By \cite[Lemma 5.4]{vdBanKuit_HC-TransformAndCuspForms} this set is finite and contained in $\fa_{\fq}^{*+}$.
It follows from \cite[Theorem 8.10 (b)]{vdBanKuit_HC-TransformAndCuspForms} that all poles of $E_{\tau}(Q:-\dotvar)$ are simple.

Let $\xi$ be the unique vector in $\fa_{\fq}^{*+}$ of unit length with respect to the Killing form.
For a meromorphic function $f: \fa_{\fq\C}^{*} \to \C$ and a point $\mu \in \fa_{\fq\C}^{*}$ we define the
residue
$$
\Res_{\gl = \mu}\; \gf(\gl) :=  \Res_{z = 0}\;  \gf(\mu + z \xi).
$$
Here, $z$ is a variable in the complex plane, and the residue
on the right-hand side is the usual residue from complex analysis, i.e., the coefficient
of $z^{-1}$ in the Laurent expansion of $z \mapsto \gf(\mu + z\xi)$ around $z = 0.$
For $\mu\in S_{Q,\tau}$ we define $\Restau(Q:\mu)=\Restau(Q:\mu:\dotvar)$ to be the function $G/H \to \Hom(\oC(\tau), \Vtau)$ given by
$$
\Restau(Q:\mu:x)(\psi)
=-\Res_{\gl=-\mu}E(Q:\psi:\gl)(x).
$$
By \cite[Theorem 8.10 (a)]{vdBanKuit_HC-TransformAndCuspForms}
\begin{equation}\label{eq Restau(Q:mu)in C_ds}
\Restau(Q:\mu)(\psi)\in\cC_{\ds}(G/H:\tau)
\qquad\big(\mu\in S_{Q,\tau},\psi\in\oC(\tau)\big).
\end{equation}

Following \cite[Section 4.1]{vdBanKuit_HC-TransformAndCuspForms}
we define for $\phi\in C_{c}^{\infty}(G/H:\tau)$ the smooth function $\It_{Q,\tau}\phi:A_{\fq}\to\oC(\tau)$ that is determined by the equation
$$
\big\langle\It_{Q,\tau}\phi(a),\psi\big\rangle
=\lim_{\epsilon\downarrow0}   \int_{\epsilon\nu+i\fa_{\fq}^{*}}
    \int_{G/H}\langle\phi(x),E_{\tau}(Q:\psi:-\bar\lambda)(x)a^{\lambda}\,dx\,d\lambda.
$$
for every $\psi\in\oC(\tau)$ and $a\in A_{\fq}$.
Here $\nu$ is any choice of element of $\fa_{\fq}^{*+}$; the definition is independent
of this choice. The map $\It_{Q,\tau}: C_{c}^{\infty}(G/H:\tau)\to C^{\infty}(A_{\fq})\otimes\oC(\tau)$ extends to a continuous map
$$
\It_{Q,\tau}: \cC(G/H:\tau)\to C^{\infty}(A_{\fq})\otimes\oC(\tau).
$$
See \cite[Proposition 7.2]{vdBanKuit_EisensteinIntegrals}.
The image of $\It_{Q,\tau}$ is contained in the tempered $\oC(\tau)$-valued functions on $A_{\fq}$ and is called the tempered term of the Harish-Chandra transform. This map has the following properties.

\begin{Prop}[{\cite[Corollaries 8.2 \& 8.11]{vdBanKuit_HC-TransformAndCuspForms}}]\label{Prop H phi=exp-pol terms +I phi}\
\begin{enumerate}[(i)]
\item Let $\phi\in\cC(G/H:\tau)$. Then for every $\psi\in\oC(\tau)$ and $a\in A_{\fq}$
\begin{equation}\label{eq langle H phi,psi rangle= sum residues + langle I phi, psi rangle}
\langle\Ht_{Q,\tau}\phi(a)-\It_{Q,\tau}\phi(a),\psi\rangle
=\sum_{\mu\in S_{Q,\tau}}a^{\mu}\int_{G/H}\big\langle\phi(x),\Restau(Q:\mu:x)(\psi)\big\rangle\,dx.
\end{equation}
\item
$\cC_{\ds}(G/H:\tau)=\ker(\It_{Q,\tau})$.
\end{enumerate}
\end{Prop}

\section{Proof of Theorem \ref{Main theorem}}\label{Section proof}
From (\ref{eq criterion ds is cusp}) it follows that {\em (ii)} implies {\em (i)} in Theorem \ref{Main theorem}. Moreover, if (\ref{eq C_ds(G/H)^K cap C_cusp(G/H)=0}) holds, then {\em (i)} implies {\em (ii)}.
It remains to prove (\ref{eq C_ds(G/H)^K cap C_cusp(G/H)=0}).

Let $Q\in\cP_{\fh}(A)$. Let further $\1_{K}$ be the trivial representation of $K$ and let $\phi\in\cC_{\ds}(G/H:\1_{K})=\cC_{\ds}(G/H)^{K}$. Then $\It_{Q,\1_{K}}\phi=0$. Hence $\Ht_{Q,\1_{K}}\phi=0$ if and only if the right-hand side of (\ref{eq langle H phi,psi rangle= sum residues + langle I phi, psi rangle}) vanishes for all $a\in A_{\fq}$ and all $\psi\in\oC(\1_{K})$. The latter is true if and only if
$$
\int_{G/H}\big\langle\phi(x),\Restriv(Q:\mu:x)(\psi)\big\rangle\,dx
=0\qquad\big(\mu\in S_{Q,\1_{K}}, \psi\in\oC(\1_{K})\big),
$$
i.e., $\Ht_{Q,\1_{K}}\phi=0$ if and only if $\phi$ is perpendicular to
$$
V_{Q}
:=\spn\big\{\Restriv(Q:\mu)(\psi):\mu\in S_{Q,\1_{K}},\psi\in\oC(\1_{K})\big\}.
$$
To show (\ref{eq C_ds(G/H)^K cap C_cusp(G/H)=0}) it thus suffices to prove the following proposition.

\begin{Prop}\label{Prop V_Q=C_ds(G/H)^K}
$V_{Q}=\cC_{\ds}(G/H)^{K}$.
\end{Prop}

To prove the proposition we will study the orthogonal projection (with respect to the inner product on $L^{2}(G/H:\1_{K})$)
$$
T_{\ds}:C_{c}^{\infty}(G/H:\1_{K})\to\cC_{\ds}(G/H:\1_{K}).
$$
To this end we first recall a formula for $T_{\ds}$.

Let the minimal $\sigma\theta$-stable parabolic $P_{0}$ be as before (see Section \ref{Section Formula for H_(Q,tau)}).
For $\lambda\in \fa_{\fq\C}^{*}$ and $\psi\in\oC(\1_{K})$ we define the Eisenstein integral $E_{\1_{K}}(\bar P_{0}:\psi:\lambda)=E(P_{0}:\psi:\lambda)$ like $E_{\tau}(Q:\psi:\lambda)$ in the previous section, but with $\tau$ and $Q$ replaced by $\1_{K}$ and $\bar P_{0}=\theta P_{0}$ respectively. Note that in order to replace $Q$ by $P_{0}$ in this construction, we need to replace the space $\oC(\tau)$ by
$$
\cA_{M_{0},2}(\tau)
:=\bigoplus_{v \in \cW}C^\infty(M_{0}/M_{0}\cap vHv^{-1}:\tau_{M_{0}}),
$$
where $\tau_{M_{0}}$ is the restriction of $\tau$ to $M_{0}\cap K$.
However, in view of \cite[Lemma 8.1]{vdBanKuit_HC-TransformAndCuspForms} applied with $vHv^{-1}$ in place of $H$, for $v\in\cW$, we have
$$
\oC(\tau)
\simeq\cA_{M_{0},2}(\tau).
$$

We normalize these Eisenstein integrals as in \cite[Section 5]{vdBanSchlichtkrull_FourierTransformOnASemisimpleSymmetricSpace} and thus we obtain the normalized Eisenstein integral
$$
E^{\circ}(\bar P_{0}:\psi:\lambda)\in C^{\infty}(G/H:\1_{K})
$$
for $\psi\in\oC(\1_{K})$ and generic $\lambda\in\fa_{\fq\C}^{*}$

We define
$$
A_{\fq}^{-}
:=\{a\in A:a^{\alpha}<1\text{for all}\alpha\in\Sigma(P_{0})\}.
$$
For $w\in\cW$ let $\delta_{w}\in \oC(\1_{K})$ be the element satisfying
$$
\langle\psi,\delta_{w}\rangle
=\psi_{w}(e)
\qquad\big(\psi\in\oC(\1_{K})\big).
$$
Observe that $\oC(\1_{K})$ is spanned by $\{\delta_{w}:w\in\cW\}$.
For $w\in\cW$ and generic $\lambda\in\fa_{\fq\C}^{*}$ we write
$$
\Phi_{w}(\lambda:\dotvar)=\Phi_{\bar P_{0},w}(\lambda:\dotvar):A_{\fq}^{-}\to\End(\C)=\C
$$
for the function introduced in \cite[Section 10]{vdBanSchlichtkrull_ExpansionsForEisensteinIntegralsOnSemisimpleSymmetricSpaces}.
From equation (53) and Remark 6.2 in \cite{vdBanSchlichtkrull_ExpansionsForEisensteinIntegralsOnSemisimpleSymmetricSpaces} it follows that $\Phi_{w}(\lambda,a)$ depends holomorphically on $\lambda$ for $\lambda\in\fa_{\fq}^{*+}+i\fa_{\fq}^{*}$. Moreover, it can be seen from (15) and Proposition 5.2 in the same article that $\Phi_{w}(\lambda:a)$ is real for $\lambda\in\fa_{\fq}^{*+}$.

Let $\Delta=\{-\alpha\}$ be the set of simple roots in $\Sigma(\bar P_{0})$. From \cite[Theorem 21.2(c)]{vdBan&Schlichtkrull_ThePlancherelDecompositionForAReductiveSymmetricSpaceI} (see also Definition 12.1) it follows that $T_{\ds}$ coincides with the operator $T_{\Delta}$ defined in equation (5.5) in \cite{vdBanSchlichtkrull_FourierInversionOnAReductiveSymmetricSpace}.
In our setting it is straight forward to rewrite this equation and thus obtain the following formula for $T_{\ds}$. For $\phi\in C_{c}^{\infty}(G/H:\1_{K})$, $w\in\cW$ and $a\in A_{\fq}^{-}$
\begin{equation}\label{eq T_ds}
T_{\ds}\phi(w^{-1}aw)
=\int_{G/H}\phi(x)\sum_{\mu\in S}
    \overline{\Res_{\lambda=\mu}\Big(\Phi_{w}(\lambda:a)E^{\circ}(\bar P_{0}:\delta_{w}:-\lambda)(x)\Big)}\,dx.
\end{equation}
Note that $T_{\ds}\phi$ is completely determined by this formula as $KA_{\fq}^{-}\cW H$ is a dense open subset of $G$.

We now compare the residues occurring in (\ref{eq T_ds}) to the residues $\Restriv(Q:\mu)$. This is done in the following lemma.

\begin{Lemma}\label{Lemma V_Q compared to kernel}
The set $S:=S_{Q,\1_{K}}$ is equal to the set of $\lambda\in\fa_{\fq}^{*+}+i\fa_{\fq}^{*}$ such that
$$
\lambda\mapsto \Phi_{w}(\lambda:a)E^{\circ}(\bar P_{0}:\delta_{w}:-\lambda)
$$
is singular at $\lambda$ for some $w\in \cW$ and $a\in A_{\fq}^{-}$. The poles which occur are simple.
Moreover, for every $\mu\in S$ there exists a constant $c_{\mu}>0$ so that for every $w\in\cW$ and $a\in A_{\fq}^{-}$
\begin{equation}\label{eq Res(Q)=C Res(bar P0)}
\Res_{\lambda=\mu}\Big(\Phi_{w}(\lambda:a)E^{\circ}(\bar P_{0}:\delta_{w}:-\lambda)\Big)
=c_{\mu}\Phi_{w}(\mu:a)\Restriv(Q:\mu)(\delta_{w}).
\end{equation}
\end{Lemma}

\begin{proof}
Let $P\in\cP(A)$ be the unique minimal parabolic subgroup contained in $P_{0}$ with $\Sigma(P)\cap\fa_{\fh}^{*}=\Sigma(Q)\cap\fa_{\fh}^{*}$. For generic $\lambda\in\fa_{\C}^{*}$ the standard intertwining operator $A\big(\sigma(P):Q:\1_{M}:\lambda\big)$ maps $C^{\infty}(Q:\1_{M}:\lambda)^{K}$ to $C^{\infty}(\sigma(P):\1_{M}:\lambda)^{K}$. Both of these spaces are $1$-dimensional. Let $\1_{Q,\lambda}\in C^{\infty}(Q:\1_{M}:\lambda)^{K}$ and $\1_{\sigma(P),\lambda}\in C^{\infty}(\sigma(P):\1_{M}:\lambda)^{K}$ be determined by
$$
\1_{Q,\lambda}(e)=\1_{\sigma(P),\lambda}(e)=1.
$$
Then the action of $A\big(\sigma(P):Q:\1_{M}:\lambda\big)$ on $C^{\infty}(Q:\1_{M}:\lambda)^{K}$ is determined by the identity
$$
A\big(\sigma(P):Q:\1_{M}:\lambda\big)\1_{Q,\lambda}
=c\big(\sigma(P),Q:\lambda\big)\1_{Q,\lambda}.
$$
Here $c:=c\big(\sigma(P),Q:\cdot\big)$ is the partial $c$-function which for $\lambda$ in the set
$$
\{\lambda\in\fa_{\C}^{*}:\Re\langle\lambda,\alpha\rangle>0\textnormal{ for all }\alpha\in\Sigma(P_{0})\cap\Sigma(Q)\}
$$
is given by the integral
\begin{equation}\label{eq c-function}
c(\lambda)
=\int_{\theta N_{P_{0}} \cap \theta N_{Q}} \1_{Q,\lambda}(\overline{n})\,d\overline{n},
\end{equation}
and for other $\lambda\in\fa_{\C}^{*}$ by meromorphic continuation. It follows from \cite[Proposition 4.4]{vdBanKuit_HC-TransformAndCuspForms} that for generic $\lambda\in\fa_{\fq\C}^{*}$
$$
E_{\1_{K}}(Q:\psi:-\lambda)
=c(\lambda+\rho_{Q,\fh})E^{\circ}\big(\bar P_{0}:\psi:-\lambda\big)
\qquad\big(\psi\in\oC(\1_{K})\big).
$$

By assumption $Q\in\cP_{\fh}(A)$, hence $\langle\rho_{Q,\fh},\alpha\rangle\geq0$ for all $\alpha\in\Sigma(Q)$. Therefore $\langle \lambda+\rho_{Q,\fh},\alpha\rangle>0$ for all $\alpha\in \Sigma(P_{0})\cap\Sigma(Q)$ if $\lambda\in\fa_{\fq}^{*+}+i\fa_{\fq}^{*}$, and thus $\lambda\mapsto c(\lambda+\rho_{Q,\fh})$ is holomorphic on $\fa_{\fq}^{*+}+i\fa_{\fq}^{*}$ and given by the integral representation (\ref{eq c-function}). Note that for $\lambda\in\fa_{\fq}^{*+}(P)$ the integrand is strictly positive, hence $c(\lambda+\rho_{Q,\fh})>0$.

Let $\mu\in S$. Since the pole of $E_{\1_{K}}(Q:\psi:-\lambda)$ at $\lambda=-\mu$ is simple and the function
$$
\lambda\mapsto\frac{\Phi_{w}(\lambda:a)}{c(\lambda+\rho_{Q,\fh})}
$$
is holomorphic on $\fa_{\fq}^{*+}+i\fa_{\fq}^{*}$, it follows that
$$
\Res_{\lambda=\mu}\Big(\Phi_{w}(\lambda:a)E^{\circ}(\bar P_{0}:\delta_{w}:-\lambda)\Big)
=\frac{\Phi_{w}(\mu:a)}{c(\mu+\rho_{Q,\fh})}\Restriv(Q:\mu)(\delta_{w}),
$$
hence (\ref{eq Res(Q)=C Res(bar P0)}) follows with $c_{\mu}=\frac{1}{c(\mu+\rho_{Q,\fh})}$.
\end{proof}

\begin{proof}[Proof of Proposition \ref{Prop V_Q=C_ds(G/H)^K}]
Since $T_{\ds}$ is the restriction to $C_{c}^{\infty}(G/H)^{K}$ of the orthogonal projection $\cC(G/H)^{K}\to\cC_{\ds}(G/H)^{K}$ (with respect to the $L^{2}$-inner product), it follows from the formula (\ref{eq T_ds}) for $T_{\ds}$ and Lemma \ref{Lemma V_Q compared to kernel} that
\begin{align*}
\cC_{\ds}(G/H)^{K}
&=\spn\big\{\sum_{\mu\in S}\Res_{\lambda=\mu}\Big(\Phi_{w}(\lambda:a)E^{\circ}(\bar P_{0}:\delta_{w}:-\lambda)\Big)
    :a\in A_{\fq}^{-},w\in\cW\big\}\\
&=\spn\big\{\sum_{\mu\in S}c_{\mu}\Phi_{w}(\mu:a)\Restriv(Q:\mu)(\delta_{w})
    :a\in A_{\fq}^{-},w\in\cW\big\}\\
&\subseteq V_{Q}.
\end{align*}
The other inclusion $V_{Q}\subseteq \cC_{\ds}(G/H)^{K}$ is a consequence of (\ref{eq Restau(Q:mu)in C_ds}).
\end{proof}

\section{Multiplicity of $K$-spherical discrete series representations}\label{Section Multiplicity}

In this final section we use the analysis that has been used for the proof of Theorem \ref{Main theorem} to prove Theorem \ref{Thm multiplicity 1}.

We begin with a lemma. If $\pi$ is a discrete series representation for $G/H$, then we write $\cC_{\pi}(G/H)$ for the closure of the span of the $K$-finite generalized matrix coefficients of $\pi$ in $\cC(G/H)$. Note that the closure of $\cC_{\pi}(G/H)$ in $L^{2}(G/H)$ decomposes into a direct sum of representation equivalent to $\pi$.

\begin{Lemma}
For every $K$-spherical discrete series representation $\pi$ of $G/H$ there exists a unique $\mu\in S_{Q,\1_{K}}$ so that
\begin{equation}\label{eq C_pi subseteq span Restriv}
\cC_{\pi}(G/H)^{K}
\subseteq\spn\big\{\Restriv(Q:\mu)(\psi):\psi\in\oC(\1_{K})\big\}.
\end{equation}
Moreover, if $\mu,\nu\in S$ and $\mu\neq \nu$, then for every $\psi,\chi\in\oC(\1_{K})$
\begin{equation}\label{eq Restrivs orthogonal}
\int_{G/H}\Restriv(Q:\mu:x)(\psi)\overline{\Restriv(Q:\nu:x)(\chi)}\,dx=0.
\end{equation}
\end{Lemma}

\begin{proof}
Let $\pi$ be a $K$-spherical discrete series representation for $G/H$. Then $\cC_{\pi}(G/H)^{K}$ is non-zero and $\cC_{\pi}(G/H)^{K}$ is canonically identified with a subspace $\cC_{\pi}(G/H:\1_{K})$ of $\cC(G/H:\1_{K})$. Let $\phi\in\cC_{\pi}(G/H:\1_{K})$. Let $\Delta_{G/H}$ and $\Delta_{A_{\fq}}$ be the Laplacian on $G/H$ and $A_{\fq}$ respectively. Since $\phi$ is a joint-eigenfunction of $\D(G/H)$, there exists a $c\in\C$ such that
$$
\Delta_{G/H}\phi=c\phi.
$$
The constant $c$ depends only on $\pi$, not on the particular choice of $\phi$.
By \cite[Lemma 8.4]{vdBanKuit_HC-TransformAndCuspForms}, the function $\Ht_{Q,\1_{K}}\phi$ satisfies
\begin{equation}\label{eq Delta H phi= gamma H phi}
\Delta_{A_{\fq}}\Ht_{Q,\1_{K}}\phi
=(c+\langle\rho_{P_{0}},\rho_{P_{0}}\rangle) \Ht_{Q,\1_{K}}\phi.
\end{equation}

It follows from Proposition \ref{Prop H phi=exp-pol terms +I phi} that $\Ht_{Q,\1_{K}}\phi$ is a finite sum of exponential functions, all with non-zero real exponents $\mu$ in the set $S_{Q,\1_{K}}$. Together with (\ref{eq Delta H phi= gamma H phi}) this implies that there exists a unique $\mu\in S_{Q,\1_{K}}$ (only depending on $\pi$, not on the function $\phi$) with $\langle\mu,\mu\rangle=c+\langle\rho_{P_{0}},\rho_{P_{0}}\rangle$, and a $\psi_{0}\in\oC(\1_{K})$  such that
$$
\Ht_{Q,\1_{K}}\phi(a)
=a^{\mu}\psi_{0}.
$$
In view of (\ref{eq langle H phi,psi rangle= sum residues + langle I phi, psi rangle}) it follows that $\phi$ is orthogonal to $\Restriv(Q:\nu)(\psi)$ for every $\nu\in S$ with $\nu\neq\mu$ and $\psi\in\oC(\1_{K})$.
We conclude that for every $K$-spherical discrete series representation $\pi$ there exists a unique $\mu\in S$ such that $\Delta_{G/H}$
$$
\cC_{\pi}(G/H)^{K}
\subseteq\Big(\bigoplus_{\nu\in S\setminus\{\mu\}}\spn\{\Restriv(Q:\nu)(\psi):\psi\in\oC(\1_{K})\}\Big)^{\perp}.
$$

For $\mu\in S$, let $D_{\mu}$ be the set of discrete series representations $\pi$, such that $\Delta_{G/H}$ acts on $\cC_{\pi}(G/H)$ by the scalar  $\langle\mu,\mu\rangle-\langle\rho_{P_{0}},\rho_{P_{0}}\rangle$. It follows from Proposition \ref{Prop V_Q=C_ds(G/H)^K} that
$$
\bigoplus_{\pi\in D_{\mu}}\cC_{\pi}(G/H)^{K}
=\Big(\bigoplus_{\nu\in S\setminus\{\mu\}}\spn\{\Restriv(Q:\nu)(\psi):\psi\in\oC(\1_{K})\}\Big)^{\perp},
$$
hence for every $\mu\in S$
\begin{align*}
\bigoplus_{\pi\in D_{\mu}}\cC_{\pi}(G/H)^{K}
&=\Big(\bigoplus_{\nu\in S\setminus\{\mu\}}\bigoplus_{\pi\in D_{\nu}}\cC_{\pi}(G/H)\Big)^{\perp}\\
&=\bigcap_{\nu\in S\setminus\{\mu\}}\bigoplus_{\chi\in S\setminus\{\nu\}}\spn\{\Restriv(Q:\chi)(\psi):\psi\in\oC(\1_{K})\}\\
&=\spn\{\Restriv(Q:\mu)(\psi):\psi\in\oC(\1_{K})\}.
\end{align*}
This proves the assertions in the proposition.
\end{proof}

\begin{proof}[Proof of Theorem \ref{Thm multiplicity 1}]
Let $\pi$ be a $K$-spherical discrete series representation.
If $|\cW|=1$, then the right-hand side of (\ref{eq C_pi subseteq span Restriv}) is $1$ dimensional, hence $\dim\cC_{\pi}(G/H)^{K}=1$ and the multiplicity with which $\pi$ occurs in the Plancherel decomposition is equal to $1$.

Now assume that $|\cW|=2$. In view of Lemma \ref{Lemma V_Q compared to kernel} we may rewrite (\ref{eq T_ds}) as
$$
T_{\ds}\phi(w^{-1}aw)
=\sum_{\mu\in S}c_{\mu}\Phi_{w}(\mu:a)
    \int_{G/H}\phi(x)\overline{\Restriv(Q:\mu:x)(\delta_{w})}\,dx,
$$
with $a\in A_{\fq}^{-}$ and $w\in\cW$. We used in the derivation of this formula that $\Phi_{w}(\mu:\dotvar)$ is real valued for $\mu\in\fa_{\fq}^{*+}$.
Since $\Restriv(Q:\mu)(\delta_{w})\in\cC_{\ds}(G/H)^{K}$, it follows in view of (\ref{eq Restrivs orthogonal}) that for $v,w\in\cW$ and $a\in A_{\fq}^{-}$
\begin{align*}
&\Restriv(Q:\mu:w^{-1}aw)(\delta_{v})\\
&\qquad=c_{\mu}\Phi_{w}(\mu:a)
    \int_{G/H}\Restriv(Q:\mu:x)(\delta_{v})\overline{\Restriv(Q:\mu:x)(\delta_{w})}\,dx.
\end{align*}
In particular it follows that there exist constants $c_{v,w}\in\C$ so that
$$
\Restriv(Q:\mu:kawh)(\delta_{v})=c_{v,w}\Phi_{w}(\mu:a)
\quad\big(k\in K, a\in A_{\fq}^{-},h\in H, v,w\in\cW \big).
$$
Let $v_{0}$ be the non-trivial element in $\cW$. Note that for every $w\in\cW$ the restricted functions $\Restriv(Q:\mu)(\delta_{e})\big|_{KA_{\fq}^{-}wH}$ and $\Restriv(Q:\mu)(\delta_{v_{0}})\big|_{KA_{\fq}^{-}wH}$  are linearly dependent.
Since the $\Restriv(Q:\mu)(\delta_{v})$ are $K$-fixed (hence $K$-finite) vectors in an irreducible subrepresentation of $L^{2}(G/H)$, they are analytic vectors and hence real analytic functions on $G/H$. It follows that $c_{v,w}$ is independent of $w\in\cW$ and thus that $\Restriv(Q:\mu)(\delta_{e})$ and $\Restriv(Q:\mu)(\delta_{v_{0}})$ are linearly dependent. Therefore, the right-hand side of (\ref{eq C_pi subseteq span Restriv}) is $1$-dimensional. This implies that $\dim\cC_{\pi}(G/H)^{K}=1$ and that $\pi$ occurs in the Plancherel decomposition of $G/H$ with multiplicity $1$.
\end{proof}

 \bibliographystyle{plain}
 \newcommand{\SortNoop}[1]{}\def\dbar{\leavevmode\hbox to 0pt{\hskip.2ex
  \accent"16\hss}d}

\def\adritem#1{\hbox{\small #1}}
\def\distance{\hbox{\hspace{0.5cm}}}
\def\apetail{@}
\def\addVdBan{\vbox{
\adritem{E.~P.~van den Ban}
\adritem{Mathematical Institute}
\adritem{Utrecht University}
\adritem{PO Box 80 010}
\adritem{3508 TA Utrecht}
\adritem{The Netherlands}
\adritem{E-mail: E.P.vandenBan{\apetail}uu.nl}
}
}
\def\addKuit{\vbox{
\adritem{J.~J.~Kuit}
\adritem{Institut f\"ur Mathematik}
\adritem{Universit\"at Paderborn}
\adritem{Warburger Stra{\ss}e 100}
\adritem{33089 Paderborn}
\adritem{Germany}
\adritem{E-mail: j.j.kuit{\apetail}gmail.com}
}
}
\def\addSchlichtkrull{\vbox{
\adritem{H.~Schlichtkrull}
\adritem{Dep. of Mathematical Sciences}
\adritem{University of Copenhagen}
\adritem{Universitetsparken 5}
\adritem{2100 K\o benhavn \O}
\adritem{Denmark}
\adritem{E-mail: schlicht{\apetail}math.ku.dk}
}
}
\mbox{}
\vfill
\hbox{\vbox{\addVdBan}\vbox{\distance}\vbox{\addKuit}\vbox{\distance}\vbox{\addSchlichtkrull}}

\end{document}